\documentclass{amsart} 

\usepackage{amsthm, amssymb, amsmath, amscd}
\usepackage{eurosym}
\usepackage{amsfonts}
\usepackage{amsmath}

\usepackage[all,cmtip]{xy}

\usepackage{indentfirst}

\newtheorem{theorem}{Theorem}[section]

\newtheorem{cor}[theorem]{Corollary}

\newtheorem{lemma}[theorem]{Lemma}

\newtheorem{prop}[theorem]{Proposition}

\theoremstyle{definition}
\newtheorem{define}[theorem]{Definition}
\newtheorem{ex}[theorem]{Example}

\newcommand{\R}{\mathbb{R}}
\newcommand{\C}{\mathbb{C}} 
\newcommand{\N}{\mathbb{N}}

\newcommand{\Z}{{\mathbb Z}}

\renewcommand{\phi}{\varphi}

\newcommand{\field}[1]{\mathbb{#1}}

\begin{document}
\title{Smale space C$^*$-algebras have nonzero projections}
\author{Robin J. Deeley}
\address{Robin J. Deeley,   Department of Mathematics,
University of Colorado Boulder
Campus Box 395,
Boulder, CO 80309-0395, USA }
\email{robin.deeley@colorado.edu}
\author{Magnus Goffeng}
\address{Magnus Goffeng\\
Department of Mathematical Sciences \\
University of Gothenburg and Chalmers University of Technology \\
Chalmers Tv\"argata 3, 412 96 G\"oteborg, Sweden}
\email{goffeng@chalmers.se}
\author{Allan Yashinski}
\address{Allan Yashinski,  Department of Mathematics, University of Maryland, College Park, MD 20742-4015, USA }
\email{ayashins@math.umd.edu}
\date{\today}
\begin{abstract}
The main result of the present paper is that the stable and unstable $C^*$-algebras associated to a mixing Smale space always contain nonzero projections. This gives a positive answer to a question of the first listed author and Karen Strung and has implications for the structure of these algebras in light of the Elliott program for simple $C^*$-algebras. Using our main result, we also show that the homoclinic, stable, and unstable algebras each have real rank zero.
\end{abstract}
\maketitle

\section{Introduction}

The starting point for this paper is the following fundamental question: when is a $C^*$-algebra stably isomorphic to a unital $C^*$-algebra? This question is well-studied, as an example see \cite{Bro}. It is equivalent to the existence of a full projection in the stabilization of the given algebra. In the simple case all nonzero projections are full. Hence this question is equivalent to the existence of a nonzero projection in the special, but important, simple case. 

The Elliott program for simple, separable, nuclear $C^*$-algebras is exhausted by three mutually exclusive cases (see \cite[Section 10]{elliotticm} or \cite[Section 2.2]{RorBook}): 
\begin{enumerate}
\item the case where no nonzero projection exists but nonzero traces exist, 
\item the case where nonzero projections and traces exist and 
\item finally the case where nonzero projections exist but nonzero traces do not. 
\end{enumerate}
Given a class of naturally occurring (potentially) classifiable $C^*$-algebras, it is interesting to ask which of the three cases can appear. 
In the unital case, the unit provides a nonzero projection excluding case (1).

The $C^*$-algebras of interest in the present paper are constructed from hyperbolic dynamical systems called Smale spaces. This class of dynamical systems includes many examples, namely, subshifts of finite type, certain solenoids \cite{Wielerpaper, Wil}, certain tiling spaces \cite{AndPut}, certain self-similiar groups, Anosov homeomorphisms (e.g., hyperbolic toral automorphisms), among others. In \cite{DS}, the first listed author and Karen Strung studied the question of whether the stable and unstable Smale space $C^*$-algebras constructed by Putnam \cite{Put} fit into classification. The present paper provides the last missing piece in \cite{DS} thus proving that the stable and unstable Smale space $C^*$-algebras fit into classification. A consequence is that $K$-theoretical invariants capture the same amount of dynamical information about a Smale space as its stable and unstable algebra.

Given any nonwandering Smale space one can associate three $C^*$-algebras, the stable, unstable and homoclinic, see \cite[Section 2]{PutFunProp}. These algebras are each separable, nuclear, stably finite and the main result of \cite{DS} states that they have finite nuclear dimension (see \cite{WinZac} for more on nuclear dimension). Moreover, Smale's decomposition theorem reduces many questions about the nonwandering case to the mixing case, see \cite[Section 2]{PutFunProp} for details. Hence we restrict to the mixing case (except for in Theorem \ref{RRZeroNonWanCase}). In the mixing case, these algebras are each simple. Since these algebras have nonzero traces \cite{Put}, they must lie in either the first or second class in the above list.  Based on examples, the first listed author and Karen Strung asked if the stable and unstable algebras always contain a nonzero projection, see \cite[Theorem 4.8 and Question 4.9]{DS}. 

The main goal of this note is to show that this is the case, see Theorem \ref{mainResult}. As a result, the stable and unstable $C^*$-algebras of a mixing Smale space are always of the second type in the list above. In particular, this implies that the Elloitt invariant for the stable and unstable algebra associated to a mixing Smale space is the following data
\[
(K_0(A), K_0(A)^+, \mathcal{D}_0(A), K_1(A), T(A), r_A: K_0(A)\times T(A) \rightarrow \R)
\]
where $A$ is the relevant $C^*$-algebra, $K_*(A)$ are the $K$-theory groups of $A$, $K_0(A)^+$ is the positive cone, $\mathcal{D}_0(A)$ is the dimension range, $T(A)$ is the cone of positive traces on $A$ and $r_A$ is the natural pairing. For more details see \cite[Section 2.2]{RorBook} (in particular page 28).  The reader can see \cite[Introduction]{DS} for more on the relationship between the classification program and Smale spaces.  

Let us describe the present paper in more detail. A Smale space consists of a compact metric space with a self-homeomorphism $\varphi : X \rightarrow X$ satisfying additional axioms. For the relevant definitions, see Subsection \ref{introtsma} and \cite{PutHom, Rue2}. Consider a (mixing) Smale space $(X, \varphi)$. The stable and unstable $C^*$-algebras associated to $(X, \varphi)$ are defined from an \'etale groupoid defined from the stable and unstable relation on $(X, \varphi)$ and a finite $\varphi$-invariant subset $P$ as in \cite{PutSpi}. The existence of the set $P$ follows from the fact that the set of periodic points is dense in $X$, see \cite{Rue}. A priori these $C^*$-algebras depend on the choice of $P$, but different choices lead to Morita equivalent $C^*$-algebras \cite{PutSpi}. Moreover, for any choice of $P$, these algebras are stable in the $C^*$-algebraic sense \cite{DY} so different choices of $P$ lead to isomorphic algebras. As mentioned above, associated to a (mixing) Smale space there is also the homoclinic $C^*$-algebra, which is unital and has unique trace.

Our proof that the stable and unstable algebras contain nonzero projections involves two main ideas. The first key idea is Putnam's notion of an s/u-bijective pair and the functorial properties of such maps. Using this, one reduces the problem to the existence of projections in the stable algebra of a mixing Smale space with totally disconnected stable sets. The second key idea is that the tensor product of the stable algebra with the unstable algebra always contains a nonzero projection, which using results in \cite{DS} leads to the existence of projections in the stable algebra in the totally disconnected stable sets case.

The proof that there exists projections in the totally disconnected stable sets case is unfortunately rather indirect. Originally our plan was to explicitly construct projections in this case. Indeed, using Wieler's characterization of this class of Smale spaces as standinary inverse limits \cite{Wielerpaper} and the associated $C^*$-algebraic results in \cite{DY}, the problem is reduced to finding projections in a particularly nice Fell algebra, see \cite[Theorem 3.17 and Theorem 5.1]{DY}. We will discuss a number of examples in Section \ref{ExSec}. For now, we note that it was the examples constructed by Farrell and Jones in \cite{FarJon} that led us to consider a more indirect approach to proving the existence of projections, see Example \ref{FarJonExample} for more details. 

Using the existence of projections and results from \cite{DS} and \cite{Ror}, we obtain more information about the $C^*$-algebras of a mixing Smale space. We have (see Corollary \ref{corwithevery}) that the homoclinic algebra has real rank zero. Also, although it was not noted explicitly in \cite{DS}, it follows from work there (along with \cite[Theorem 6.7]{Ror}) that the homoclinic algebra has stable rank one. Furthermore, the stable and unstable algebras also have real rank zero (see Corollary \ref{corwithevery}), stable rank one (via \cite[Theorem 6.7]{Ror} and \cite[Theorem 3.6]{Rie}) and are stably isomorphic to simple unital $C^*$-algebras that are approximately subhomogeneous (see \cite[Theorem 4.8]{DS}). The fact that these algebras have real rank zero and stable rank one implies that their Elliott invariant can be simplified, see \cite[Section 2.2]{RorBook}. Many of these results generalize to the case of a nonwandering Smale space using Smale's decomposition theorem, see for example Theorem \ref{RRZeroNonWanCase}.

The structure of the paper is as follows. In Section \ref{prelimsec} we recall some relevant known results about projections (Subsection \ref{projsubsec}), Smale spaces (Subsection \ref{introtsma}), the associated groupoids/$C^*$-algebras (Subsection \ref{grpCstarSec}), some examples of Smale spaces and their $C^*$-algebras (Subsection \ref{ExSec}), traces on said $C^*$-algebras (Subsection \ref{tracialsubsection}) and the Morita equivalence $S\otimes U\sim_M H$ (Subsection \ref{moritasubssec}). Next the special case of the existence of projections in the stable algebra of a mixing Smale space with totally disconnected stable sets is considered in Section \ref{specialCaseSec}. In Section \ref{projinsmale} we prove the main result (Theorem \ref{mainResult}) stating that Smale space $C^*$-algebras of mixing Smale spaces always admit a full projection and hence are stably unital. In the final section of the paper, the proof of real rank zero is presented. All $C^*$-algebras in this paper are assumed to be separable. In particular, $\mathbb{K}$ denotes the compact operators on a separable, infinite-dimensional Hilbert space.

\section*{Acknowledgments}
The authors thank  Ian Putnam and Karen Strung for interesting and insightful discussions. We also thank the referee for a number of useful suggestions. The second listed author was supported by the Swedish Research Council Grant 2015-00137 and Marie Sklodowska Curie Actions, Cofund, Project INCA 600398. The authors thank the University of Colorado (Meyer Fund), the University of Gothenburg and the University of Hawaii for facilitating this collaboration.

\section{Preliminaries}
\label{prelimsec}

\subsection{Projections}
\label{projsubsec}

We recall some basic facts about projections in $C^*$-algebras. We say that a $C^*$-algebra is stably unital if it is stably isomorphic to a unital $C^*$-algebra. We say that a $C^*$-algebra $A$ admits a nonzero projection if for some $N$, there is a nonzero projection in $M_N(A)$. Equivalently, $A$ admits a nonzero projection if there is a nonzero projection in the stabilization $A\otimes \mathbb{K}$.  Thus, the property of admitting a nonzero projection is preserved under stable isomorphism and Morita equivalence.  Note that every unital $C^*$-algebra admits a nonzero projection.

\begin{lemma}
\label{stablyuni}
Let $A$ be a $C^*$-algebra and $p\in A$ a projection. Then the unital $C^*$-algebra $pAp$ is stably isomorphic to the $C^*$-algebra $\overline{ApA}$. Moreover, if $A$ is simple then $A$ is stably isomorphic to a unital $C^*$-algebra if and only if $A$ admits a nonzero projection. 
\end{lemma}

\begin{proof} 
The first part is clear from the fact that $pA$ defines a Morita equivalence $pAp\sim_M  \overline{ApA}$. We remark that $\overline{ApA}$ is a closed two-sided ideal in $A$, and the first part of the lemma implies that $\overline{ApA}$ is stably unital. For the second part, we can assume without loss of generality that $A$ is stable.  Then it follows from the first part because $\overline{ApA}\lhd A$ being a closed nonzero two-sided ideal in a simple $C^*$-algebra implies $\overline{ApA}=A$.
\end{proof}

%

\begin{lemma}
\label{directlimitsandproj}
Suppose that the $C^*$-algebra $A$ is the direct limit of a direct system $(A_\alpha)_{\alpha\in I}$ of $C^*$-algebras with injective structure maps. Then $A$ admits a nonzero projection if and only if there is an $\alpha$ such that $A_\alpha$ admits a nonzero projection.  
\end{lemma}

\begin{proof}
If $A_\alpha$ admits a nonzero projection, it is clear that $A$ also does because the structure maps are injective. Conversely, assume that there is a nonzero projection $p\in A$.  The algebraic direct limit $\mathcal{A}\subseteq A$ of $(A_\alpha)_{\alpha\in I}$ is a local $C^*$-algebra (see \cite[Definition 3.1.1 and Theorem 3.3.2]{BlaKth}) so by a standard approximation and functional calculus argument, e.g. see \cite[Proposition 4.5.1]{BlaKth}, there is a nearby nonzero projection $p_0 \in A_\alpha$ for some $\alpha$.
The same argument generalizes to the case that $p$ belongs to a matrix algebra $M_N(A)=\varinjlim M_N(A_\alpha)$. 
\end{proof}

We recall two definitions from $C^*$-algebra theory:
\begin{enumerate}
\item A $C^*$-algebra is called  subhomogeneous (see \cite[page 62]{RorEncBook}) if it is isomorphic to a $C^*$-subalgebra of $M_l(C_0(Y))$ for some $l\in \N$ and some locally compact Hausdorff space $Y$.
\item A $C^*$-algebra is called approximately subhomogeneous (see \cite[page 62]{RorEncBook}) if it is a direct limit of subhomogeneous $C^*$-algebras.
\end{enumerate}

\begin{theorem} 
\label{tensorSubHom}
Suppose $A$ and $B$ are $C^*$-algebras with $B$ an approximately subhomogeneous $C^*$-algebra. If $A\otimes B$ contains a nonzero projection, then $A$ admits a nonzero projection.
\end{theorem}

\begin{proof}
Since the quotient of a subhomogeneous $C^*$-algebra is also subhomogeneous, we can assume that the connecting maps in the direct limit are injective. By Lemma \ref{directlimitsandproj} and the definition of subhomogeneous, $A\otimes B_k$ contains a nonzero projection where $B_k$ is isomorphic to a $C^*$-subalgebra of $M_l(C_0(X_k))$ for some $l\in \N$ and some locally compact Hausdorff space $X_k$. Hence $A\otimes M_l(C_0(X_k))$ contains a nonzero projection. We note that $A\otimes M_l(C_0(X_k)) \cong C_0(X_k, M_l(A))$. Since $q\neq 0$, there exists $x\in X_k$ such that $q(x)\in M_l(A)$ is a nonzero projection. This completes the proof.
\end{proof}

\subsection{Smale spaces}
\label{introtsma}

The theory of Smale spaces and $C^*$-algebras was initiated by Ruelle \cite{Rue, Rue2} and developed further by Putnam and coauthors \cite{PutSpi,Put, PutHom}. We recall the basic definitions and constructions, and refer the reader to the above mentioned sources for details. 

\begin{define}[Section 2.1, \cite{PutHom}]
\label{SmaSpaDef}
A Smale space $(X,\varphi)$ consists of a compact metric space $(X, d)$ and a homeomorphism $\varphi: X \to X$ such that there exists constants $ \epsilon_{X} > 0, 0<\lambda < 1$ and a continuous partially defined map:
\[ [ \, \cdot \, , \, \cdot \, ]: \{(x,y) \in X \times X \mid d(x,y) \leq \epsilon_{X}\} \to X,\qquad (x,y) \mapsto [x,y] \]
satisfying the axioms listed in \cite[Section 2.1]{PutHom}.
\end{define}

The ``bracket map" $[ \, \cdot \, , \, \cdot \, ]$ in the definition of a Smale space is unique (provided it exists). 
Suppose $(X, \varphi)$ is a Smale space, and let $x \in X$, $0<\epsilon\le \epsilon_X$, and $Y, Z \subset X$. The dynamical structure is studied by means of the following sets:
\begin{enumerate}
\item $X^s(x, \varepsilon)  :=  \left\{ y \in X \mid d(x,y) < \varepsilon, [y,x]=x \right\},$
\item $X^u(x, \varepsilon)  :=  \left\{ y \in X \mid d(x,y) < \varepsilon, [x,y]=x \right\}, $
\item $X^s(x)  :=  \left\{ y\in X \mid \lim_{n \rightarrow + \infty} d(\varphi^n(x), \varphi^n(y)) =0 \right\}, $
\item $X^u(x)   :=  \left\{ y\in X \mid \lim_{n \rightarrow - \infty} d(\varphi^n(x), \varphi^n(y)) =0 \right\},$
\item $X^s(Z) := \cup_{x\in Z} X^s(x)$,
\item $X^u(Z):=\cup_{x\in Z} X^u(x)$,
\item $X^h(Y, Z) := X^s(Y) \cap X^u(Z)$ and
\item $X^h(Y):= X^h(Y, Y)$.
\end{enumerate}
One writes $x \sim_s y$ when $y \in X^s(x)$ and $x \sim_u y$, when $y \in X^u(x)$. It follows from the Smale space axioms that $X^s(x,\epsilon) \subseteq X^s(x)$.  The space $X^s(x)$ admits a metrizable locally compact Hausdorff topology defined from equipping the local stable sets $X^s(x,\epsilon)\subseteq X$ with the subspace topology and using $X^s(y,\epsilon)\subseteq X^s(x)$ (where $0< \epsilon \le \epsilon_X$ and $y$ varies over the elements in the stable equivalence class of $x$) as a base for the topology on $X^s(x)$. Note that this topology is not the same as the subspace topology that $X^s(x)$ inherits from $X$.  The analogous construction using local unstable sets topologizes $X^u(x)$. To avoid certain trivial cases, $X$ is always assumed to be infinite. 

\begin{define} 
\label{MixDef}
A Smale space $(X, \varphi)$ is mixing if for any ordered pair of nonempty open sets, $(U,V)$, there exists $N\in \N$ such that for any $n\ge N$, $\varphi^n(U) \cap V \neq \emptyset$.
\end{define}

A factor map $f:(Y,\psi)\to (X,\phi)$ is a continuous surjective map compatible with the dynamics (i.e. $f\circ \psi=\phi\circ f$). Such an $f$ is said to be \emph{$s$-bijective} if $f$ induces a bijection on the stable sets $Y^s(y)\to X^s(f(y))$ (see \cite[Definition 2.5.5]{PutHom}). Similarly, $f$ is u-bijective if the induced mapping $Y^u(y)\to X^u(f(y))$ is a bijection. 

With the exception of Theorem \ref{RRZeroNonWanCase}, all Smale spaces in this paper are assumed to be mixing. In particular, \cite[Theorem 2.5.8]{PutHom} implies that the notions of s-bijective (resp. u-bijective) are equivalent to Fried's notions of s-resolving (resp. u-resolving). 



\subsection{Groupoids and $C^*$-algebras associated to a Smale space} 
\label{grpCstarSec} 

Given a (mixing) Smale space $(X, \varphi)$, following the constructions in \cite{Put, PutSpi, Rue2}, one can form three groupoid $C^*$-algebras. By definition the homoclinic groupoid is
\[ G_h(X, \varphi):= \{ (x,y) \: | \: x\sim_s y \hbox{ and } x\sim_u y\}, \]
with unit space $X$. There is an \'etale topology on this groupoid, see \cite{Put, Rue2}. Moreover, $G_h(X, \varphi)$ is also amenable and the associated $C^*$-algebra is denoted by $H(X, \varphi)$ or by $H$ if the Smale space is clear from the context. This algebra is called the homoclinic algebra. 

The stable and unstable groupoids associated to $(X, \varphi)$ are constructed as follows. Let $P$  be a finite $\varphi$-invariant set of periodic points of $(X, \varphi)$. Define the stable groupoid
\[
G_s(P):= \{ (x, y) \in X \times X  \mid x \sim_s y \hbox{ and } x, y \in X^u(P)\}.
\]
Using results from \cite{PutSpi}, $G_s(P)$ can be given a locally compact \'etale topology over the unit space $X^u(P)$. Moreover, still following \cite{PutSpi}, one constructs the $C^*$-algebra algebra associated to this groupoid (denoted by $C^*(G_s(P))$) by noticing that the groupoid is amenable and taking the closure of an explicit representation of the compactly supported functions, $C_c(G_s(P))$, on the Hilbert space $l^2(X^h(P))$. If the Smale space and set of periodic points is clear from the context, the stable algebra is denoted by $S$. On the other hand, when working with a number of Smale spaces we include the particular Smale space in the notation; that is we denote the associated stable groupoid by $G_s(X, \varphi;P)$ and the associated $C^*$-algebra by $C^*(G_s(X, \varphi;P))$. 

A quick way to define the unstable groupoid is via $G_u(X, \varphi;P):=G_s(X, \varphi^{-1};P)$. However, we note that 
\[ G_u(X, \varphi, P)=\{ (x, y) \in X \times X  \mid x \sim_u y \hbox{ and } x, y \in X^s(P)\}. \]
The unstable $C^*$-algebra is denoted by $C^*(G_u(X, \varphi; P))$ (or by simply $U$). By construction, the stable and unstable algebra each have a representation on the Hilbert space $l^2(X^h(P))$. 



\subsection{Examples of Smale spaces and their associated $C^*$-algebras}
\label{ExSec}

A number of examples of Smale spaces are discussed. Our choice of examples is in part based on the fact that the stable algebra associated to a Smale space with totally disconnected stable sets plays an important role in the present paper. All the examples except the last have totally disconnected stable sets.

Although, our discussion of each example is rather brief we hope to give some context to the main problems considered here, namely the existence of projections and real rank zero for Smale space $C^*$-algebras. 

\begin{ex}
Two-sided subshifts of finite type provide a prototypical example of Smale spaces. The reader can see \cite{LM} for details on subshifts of finite type. The stable, unstable and homoclinic $C^*$-algebra associated to subshifts of finite type are each approximately finite (AF) \cite[pages 25--26]{Put} and hence each have real rank zero as well as admitting an abundance of projections.  
\end{ex}

\begin{ex}
Take $Y=S^1 \subseteq \C$ with the arc length metric, rescaled so that the total circumference is one.  For a fixed integer $n > 1$, define
$g: S^1 \rightarrow S^1$ via $z \mapsto z^n$. Then 
\[
X:= \varprojlim (Y, g) = \{ (y_n)_{n\in \N} = (y_0, y_1, y_2, \ldots ) \: | \: g(y_{i+1})=y_i \hbox{ for each }i\ge0 \}
\]
with $\varphi$ defined via
\[
(y_0, y_1, y_2, \ldots )  \mapsto (g(y_0), g(y_1), g(y_2), \ldots)
\]
is a mixing Smale space. The stable algebra in this case is a Bunce--Deddens algebra tensored with the compact operators. For details in the case $n=2$, see page 28 of \cite{Put}. It follows that the stable algebra was known to have real rank zero. 
\end{ex}

\begin{ex}
The construction in the previous example can be generalized in a number of ways. For example, one natural generalization is as follows. As input we take $(Y,g)$ where $Y$ is a compact metric space and $g:Y \rightarrow Y$ is an expansive surjective local homeomorphism. The output is 
\[
X:= \varprojlim (Y, g) = \{ (y_n)_{n\in \N} = (y_0, y_1, y_2, \ldots ) \: | \: g(y_{i+1})=y_i \hbox{ for each }i\ge0 \}
\]
with $\varphi$ defined via
\[
(y_0, y_1, y_2, \ldots )  \mapsto (g(y_0), g(y_1), g(y_2), \ldots).
\]
If $Y$ is connected and $g$ has a dense orbit, then $(X, \varphi)$ is a mixing Smale space. Moreover, the stable equivalence relation is related to the inverse limit in a nice way:  $(y_n)_{n\in \N} \sim_s (z_n)_{n\in \N}$ if and only if there exists $k\in \N$ such that $g^k(y_0)=g^k(z_0)$.

Based on this, we define an open subrelation of the stable groupoid by taking $(y_n)_{n\in \N} \sim_0 (z_n)_{n\in \N}$ if and only if $y_0=z_0$. For details see \cite[Theorem 3.17 and Example 5.3]{DY} in particular to show that $\sim_0$ is open it is important that $g$ is a local homeomorphism. The $C^*$-algebra associated to $\sim_0$ is isomorphic to $C(Y)\otimes \mathbb{K}$. Hence $C(Y)\otimes \mathbb{K}$ is a subalgebra of the stable algebra. This construction provides many explicit projections in the stable algebra of this class of examples; one takes $1_Y \otimes p$ where $p$ is an projection in $\mathbb{K}$. It is worth noting that the construction of these projections is related to the global structure of the dynamical system. Namely, the map $pX \rightarrow Y$, $(y_n)_{n\in \N} \mapsto y_0$ gives $X$ the structure of a fiber bundle over $Y$ with Cantor set fibers.
\end{ex}

The construction in the previous example includes a large class of Smale spaces with totally disconnected stable sets, such as the self-similar groups considered in \cite{Nek} (also see \cite{DGMW} for more examples). Nevertheless there are many Smale spaces with totally disconnected stable sets for which one must weaken the conditions on the map $g$. In particular, Wieler \cite{Wielerpaper} has shown that every mixing Smale space with totally disconnected stable sets is a solenoid. That is, the space $X$ is obtained via a stationary inverse limit construction as in the previous example. However, the map $g$ is in general not a local homeomorphism rather it is required to satisfy two natural axioms \cite[page 2068]{Wielerpaper}. Wieler's work builds on work of Williams \cite{Wil}. The next two classes of examples fit within Wieler's framework but cannot be constructed as a solenoid where $g$ is taken to be a local homeomorphism. Perhaps most importantly, the relation $\sim_0$ as defined in the previous example is a subrelation but is not open, see \cite[Section 3]{DY} for a detailed discussion of this fact.

\begin{ex}
Results in \cite{AndPut} link tilings space theory with Smale space theory. In particular, the construction of the relevant inverse limit in this case is given in \cite[Section 4]{AndPut}. Computations of the $K$-theory groups of the stable algebra for a number of tiling space examples can be found in \cite{Gon2, GonRamSol}. However, before the present paper, it was not known that the stable and homoclinic algebras associated to such tilings always have real rank zero. It was known that the unstable algebra always has real rank zero using different methods than the ones used in the present paper.
\end{ex}

\begin{ex} \label{FarJonExample}
In \cite{FarJon}, Farrell and Jones constructed very interesting examples that fit within Wieler's framework and hence give mixing Smale spaces with totally disconnected stable sets. At present the $C^*$-algebras associated to these examples are poorly understood. For example, their $K$-theory has not been computed. An important issue in understanding these $C^*$-algebras is the failure of these dynamical systems to be fiber bundles over a manifold with Cantor set fibers, see \cite[Corollary 0.3]{FarJon}. This lack of global structure makes the construction of projections difficult. Indeed, before this paper it was not known whether the stable algebra associated to these examples even had non-zero projections.
\end{ex}

\begin{ex}
Let $X$ be the $n$-fold cartesian product of circles (i.e., $X=S^1\times S^1 \times \cdots \times S^1$) and $A$ be an $n$ by $n$ matrix with integer entries that satisfies the following:
\begin{enumerate}
\item $|\det(A)| =1$;
\item no eigenvalue of $A$ has modulus one.
 \end{enumerate}
A specific example when $n=2$ is
\[
 A = \left( \begin{array}{cc} 1 & 1 \\ 1 & 0  \end{array} \right).
\]
The system $(X, \varphi)$ where $\varphi$ is given by multiplication by $A$ is a mixing Smale space and this class of examples are called hyperbolic toral automorphisms. They are examples of Anosov diffeomorphism, see for example \cite[Section 2]{Put}.

For hyperbolic toral automorphisms with $n=2$ the associated $C^*$-algebras are (up to tensoring with the compacts in the case of the stable and unstable algebras) irrational rotation $C^*$-algebras, see page 26--27 in \cite{Put}. Hence it was known that they have real rank zero. However, in the full generality of Anosov diffeomorphisms this was not known but now follows from the main results of the present paper. 
\end{ex}

\subsection{Traces on $S$, $U$ and $H$}
\label{tracialsubsection}

The algebras $S$, $U$ and $H$ have traces defined in \cite{Put}. We shall denote these by $\tau_S$, $\tau_U$ and $\tau_H$, respectively. The trace $\tau_H$ is a tracial state while $\tau_S$ and $\tau_U$ are positive lower semicontinuous tracial weights. 

Let us briefly recall their constructions. Following \cite{Put}, we let $\mu$ denote the Bowen measure on $X$. The Bowen measure is the unique $\varphi$-invariant probability measure that maximizes entropy. The tracial state $\tau_H$ is defined on $f\in C_c(G_h)$ as 
$$\tau_H(f):=\int_X f(x,x)\ \mathrm{d}\mu(x),$$
and extended to $H$ by continuity. For details, see \cite[Theorem 3.3]{Put}. The state $\tau_H$ is the unique tracial state on $H$, see \cite{Hou}.

The traces $\tau_S$ and $\tau_U$ will be constructed as in \cite{tracekiller}. By \cite[Theorem 1.1]{tracekiller}, there are for any $x\in X$ measures $\mu_s^x$ defined on $X^s(x)$ and $\mu_u^x$ defined on $X^u(x)$ such that for $0<\epsilon<\epsilon_X$,
\begin{equation}
\label{productform}
\mu([B,C])=\mu_s^x(B)\mu_u^x(C),
\end{equation}
for all Borel sets $B\subseteq X^s(x,\epsilon)$ and $C\subseteq X^u(x,\epsilon)$. For two finite $\varphi$-invariant sets of periodic points $P,Q\subseteq X$, we define $\mu_s^Q:=\sum_{q\in Q}\mu_s^q$ and $\mu_u^P:=\sum_{p\in P}\mu_u^p$. They are measures on $X^s(Q)$ and $X^u(P)$, respectively. The measures $\mu_s^Q$ and $\mu_u^P$ will by \cite[Theorem 1.1]{tracekiller} satisfy the transformation rules 
$$\varphi^* \mu_s^Q=\lambda^{-1}\mu_s^Q\quad\mbox{and}\quad \varphi^* \mu_u^P=\lambda\mu_u^P,$$
where $\log(\lambda)$ is the topological entropy of $(X,\varphi)$. The traces $\tau_S$ and $\tau_U$ are defined as follows. For $f\in C_c(G_s(P))$ and $f'\in C_c(G_u(Q))$ we define 
$$\tau_S(f):=\int_{X^u(P)} f(x,x)\ \mathrm{d}\mu_u^P(x) \quad \mbox{and}\quad \tau_U(f'):=\int_{X^s(Q)} f'(x,x)\ \mathrm{d}\mu_s^Q(x).$$
For details, see \cite[Theorem 1.2]{tracekiller}.

\subsection{Morita equivalences and projections}
\label{moritasubssec}

By \cite[Theorem 3.1]{Put}, $S\otimes U$ is Morita equivalent to $H$. Later we will need to relate the tracial state $\tau_H$ to $\tau_S\otimes \tau_U$. In order to do so, the Morita equivalence $S\otimes U\sim_M H$ needs to be specified. Following \cite[Theorem 3.1]{Put}, we define the space
$$Z(P,Q):=\{(x,x',y)\in X^u(P)\times X^s(Q)\times X: x\sim_s y\sim_u x'\}.$$
The space $Z(P,Q)$ carries a left action of $G_s(P)\times G_u(Q)$ and a right action of $G_h$ defined from 
$$((x_1,x_2),(x_1',x_2')).(x_2,x_2',y_1).(y_1,y_2)=(x_1,x_1',y_2),$$
for $((x_1,x_2),(x_1',x_2'))\in G_s(P)\times G_u(Q)$, $(x_2,x_2',y_1)\in Z(P,Q)$ and $(y_1,y_2)\in G_h$. We topologize $Z(P,Q)$ so that the actions are proper and the mappings 
\begin{align*}
G_s(P)\times G_u(Q)\backslash &Z(P,Q)\to X,\quad (x,x', y)\mapsto y,\quad\mbox{and}\\
&Z(P,Q)/G_s\to X^u(P)\times X^s(Q),\quad (x,x',y)\mapsto (x,x'),
\end{align*}
are homeomorphisms. The space $Z(P,Q)$ is a groupoid Morita equivalence and $C_c(Z(P,Q))$ can be completed into an imprimitivity module ${}_{S\otimes U}C^*(Z(P,Q))_H$ by \cite[Theorem 2.8]{MRW}. 

Let us describe the projection $p\in S\otimes U$ corresponding to $1\in H$ under Morita equivalence. Following \cite[Section 5]{KPWduality}, we take $\epsilon\in (0,\epsilon_X'/2]$ (see \cite[Lemma 2.3]{KPWduality} for the definition of $\epsilon_X'$) and consider an $\epsilon$-partition $(\mathcal{F},\mathcal{G})$ of $X$ (see \cite[Definition 5.1]{KPWduality}). Here $\mathcal{G}=\{g_1,\ldots, g_K\}\subseteq X^h(P,Q)$ is a set of distinct elements and $\mathcal{F}=\{f_1,\ldots, f_K\}\subseteq C(X)$ is a partition of unity such that for any $j$, $f_j$ is supported in $B(g_k,\epsilon/2)$. For $i,j=1,\ldots, K$, we define open sets $V_{ij}\subseteq G^s(P)\times G^u(Q)$ as the set of $((x_1,x_2),(x_1',x_2'))$ from the set
$$X^u(g_i,\epsilon)\times X^u(g_j,\epsilon)\times X^s(g_i,\epsilon)\times X^s(g_j,\epsilon)$$
such that $[x_1,x_1']=[x_2,x_2']$. By \cite[Section 5]{KPWduality}, the collection of sets $(V_{ij})_{i,j=1}^K$ is pairwise disjoint. Moreover, \cite[Lemma 5.3]{KPWduality} defines the projection $p\in S\otimes U$ by 
$$p((x_1,x_2),(x_1',x_2')):=\sum_{i,j=1}^K \chi_{V_{ij}}((x_1,x_2),(x_1',x_2'))f_i([x_1,x_1'])f_j([x_2,x_2']),$$
for $(x_1,x_2)\in G^s(P)$ and $(x_1',x_2')\in G^u(Q)$.

\section{Projections in the totally disconnected stable sets case} \label{specialCaseSec}

\begin{lemma}
If $(X,\varphi)$ is a mixing Smale space with totally disconnected stable sets, then its unstable algebra, $U$, contains a nonzero projection. Moreover, if $p\neq 0$ is a projection in $U$, then $pUp$ is an approximately subhomogeneous $C^*$-algebra.
\end{lemma}
\begin{proof}
The unit space of the unstable groupoid is $X^s(P)$, which is locally compact and totally disconnected by assumption. Hence it contains a nonempty compact open subset and the characteristic function associated to any nonempty compact open subset is a nonzero projection in $U$. 

The second part follows from the proof of \cite[Theorem 4.8]{DS} but we give the details. The $C^*$-algebra $pUp$ is separable, simple, unital, has finite nuclear dimension and satisfies the UCT. Hence by \cite[Theorem 4.3]{EllGonLinNiu:ClaFinDecRan} and \cite[Theorem A]{TWW}, $pUp$ fits within the Elliott classification program. It follows from \cite{Ell:invariant} that $pUp$ is an approximately subhomogeneous $C^*$-algebra.
\end{proof}

\begin{theorem} 
\label{totDisConCase}
If $(X,\varphi)$ is a mixing Smale space with totally disconnected stable sets, then its stable algebra contains a nonzero projection.
\end{theorem}
\begin{proof}
Let $p$ be a nonzero projection in $U$. Then $U$ is Morita equivalent to $pUp$. By \cite[Theorem 3.1]{Put}, $S\otimes U$ is Morita equivalent to $H$. Using the fact that $S$ is $C^*$-stable, $S\otimes pUp \cong H \otimes \mathbb{K}$. Since $H$ is unital it follows that $S\otimes pUp$ contains a nonzero projection. Using the previous lemma, we can now apply Theorem \ref{tensorSubHom}, which implies that $S\otimes \mathbb{K}$ contains a nonzero projection. The result follows since $S$ is $C^*$-stable.
\end{proof}

\section{Projections in Smale space $C^*$-algebras}
\label{projinsmale}

\begin{define} (Definition 2.6.2 in \cite{PutHom}) \\
An s/u-bijective pair for a Smale space $(X, \varphi)$ is $(Y, \psi, \pi_s, Z, \zeta, \pi_u)$ where
\begin{enumerate}
\item $(Y, \psi)$ and $(Z, \zeta)$ are Smale spaces such that the stable sets of $Z$ are totally disconnected and the unstable sets of $Y$ are totally disconnected.
\item $\pi_s : (Y, \psi) \rightarrow (X, \varphi)$ and $\pi_u: (Z, \zeta) \rightarrow (X, \varphi)$ are factor maps with $\pi_s$ s-bijective and $\pi_u$ u-bijective.
\end{enumerate}
\end{define}

Building on \cite[Theorem 2.6.3]{PutHom}, Amini, Putnam and Saeidi  proved the following:

\begin{theorem} 
\label{suBijectivePair} (see Theorem 2.6 and Proposition 4.8 in \cite{AmPuSa}) \\
Suppose $(X, \varphi)$ is a mixing Smale space. Then, there exists an s/u-bijective pair $(Y, \psi, \pi_s, Z, \zeta, \pi_u)$ for $(X, \varphi)$ such that $(Y, \psi)$ and $(Z, \zeta)$ are mixing.
\end{theorem}

\begin{theorem} (special case of the u-bijective version of \cite[Corollary 3.6]{PutFunProp}) \label{induceMapThm} \\
Suppose that $\pi: (Z, \zeta) \rightarrow (X, \varphi)$ is a u-bijective map between two mixing Smale space and $P$ is a finite $\zeta$-invariant subset of $Z$. Then $G_s(Z, \zeta;P)$ is an open subgroupoid of $G_s(X, \varphi;\pi(P))$ and hence there is an induced nonzero $*$-homomorphism 
\[ \pi_* : C^*(G_s(Z, \zeta;P)) \rightarrow C^*(G_s(X, \varphi;\pi(P))). \]
\end{theorem}

We can prove the existence of nonzero projections in full generality. 

\begin{theorem} 
\label{mainResult}
If $(X, \varphi)$ is a mixing Smale space, then its stable and unstable $C^*$-algebras have nonzero projections. In particular, these projections are full. 
\end{theorem}

\begin{proof}
We initiate the proof by remarking that the involved $C^*$-algebras are simple and stable, so they contain a full projection if and only if they admit a nonzero projection. Since the unstable algebra of $(X, \varphi)$ is the stable algebra of $(X, \varphi^{-1})$ we need only prove that the stable algebra contains a nonzero projection. Fix an s/u-bijective pair $(Y, \psi, \pi_s, Z, \zeta, \pi_u)$ for $(X, \varphi)$ as in Theorem \ref{suBijectivePair}. In particular, $\pi_u: (Z, \zeta) \rightarrow (X, \varphi)$ is u-bijective, the stable sets of $(Z, \zeta)$ are totally disconnected and $(Z, \zeta)$ is mixing. 

Fix a finite set of $\zeta$-invariant points, $P$, in $Z$ and form $C^*(G_s(Z, \zeta; P))$. By Theorem \ref{totDisConCase}, $C^*(G_s(Z, \zeta;P))$ has a nonzero projection. It is also a simple $C^*$-algebra. By Theorem \ref{induceMapThm}, there is nonzero $*$-homomorphism
\[
(\pi_u)_* : C^*(G_s(Z, \zeta; P)) \rightarrow C^*(G_s(X, \varphi; \pi(P)))
\]
Since $C^*(G_s(Z, \zeta; P))$ is simple and this map is nonzero, it is injective. Hence, $C^*(G_s(X, \varphi; \pi(P)))$ has a nonzero projection because $C^*(G_s(Z, \zeta; P))$ does. The result now follows, since changing the finite set of invariant points in the definition of the stable algebra does not affect the isomorphism class of the $C^*$-algebra by \cite[Appendix]{DY}.
\end{proof}

The results in \cite{DS} show that the stable and unstable algebra of a Smale space have finite nuclear dimension (and hence are quasi-diagonal). In addition, assuming existence of nonzero projections, they are proven to be approximately subhomogeneous.  From these results in \cite{DS}, Theorem \ref{mainResult} and Lemma \ref{stablyuni} we deduce the next corollary. 
We are additionally using the fact that the $C^*$-algebras associated with a mixing Smale space are simple (see \cite[Theorem 1.3]{PutSpi}) and stable (see \cite{DY}).

\begin{cor}
\label{corwithevery}
If $(X, \varphi)$ is a mixing Smale space, there are simple unital $C^*$-algebras $A_s(X,\phi,P)$ and $A_u(X,\phi,Q)$ and isomorphisms
$$C^*(G_s(X,\phi,P))\cong A_s(X,\phi,P)\otimes \mathbb{K}\quad\mbox{and}\quad C^*(G_u(X,\phi,Q))\cong A_u(X,\phi,Q)\otimes \mathbb{K}.$$
The unital $C^*$-algebras $A_s(X,\phi,P)$ and $A_u(X,\phi,Q)$ are approximately subhomogeneous, satisfy the UCT, have finite nuclear dimension (and are quasi-diagonal).
\end{cor}


\section{Real rank zero}

In this section, we discuss the traces on the stable and unstable algebras and study real rank zero in the context of Smale space $C^*$-algebras. The following facts will be used:
\begin{enumerate}
\item It follows from Perron--Frobenius theory (see in particular \cite[Theorem 4.3.1]{LM}) and Bowen's theorem (see in particular \cite[Theorem 33]{Bow1}) that the topological entropy of a mixing infinite Smale space is strictly greater than zero. Hence if $\log(\lambda)$ is the entropy then $\lambda>1$. 
\item The homoclinic algebra of a mixing Smale space has a unique tracial state $\tau_H$, see \cite{Hou} and Subsection \ref{tracialsubsection} above.
\item In \cite{tracekiller,Put}, lower semi-continuous densely defined traces $\tau_S$ and $\tau_U$ are defined on the stable and unstable algebras. See more in Subsection \ref{tracialsubsection} above.
\item Any lower semi-continuous densely defined trace on a $C^*$-algebra is finite on its Pedersen ideal (i.e., the Pedersen ideal forms a common domain of definition for such traces) see \cite[page 73]{ComZet}.
\item In particular, the previous fact implies that a lower semi-continuous densely defined trace on a unital $C^*$-algebra is a trace (i.e., finite on the entire algebra).
\end{enumerate}

\begin{prop} 
\label{uniqueTrace}
The $C^*$-algebra $S$, $U$ and $S\otimes U$ each have a unique (up to positive scaling) positive lower semi-continuous densely defined trace. Moreover, if $A$ is any of these three algebras and $p \in A$ is a nonzero projection, then the unital algebra $pAp$ has a unique tracial state.
\end{prop}

\begin{proof}
We begin with $S\otimes U$. By \cite[Theorem 3.1]{Put}, $S\otimes U$ is Morita equivalent to $H$. Since $H$ (which is unital) has a unique trace, it follows that $S\otimes U$ (which is non-unital) has a unique (up to scaling) lower semi-continuous densely defined trace by \cite[Proposition 2.2 and Corollary 2.4]{ComZet}. Moreover, this trace is given by $\tau_S \otimes \tau_U$.

Next suppose $\tau_1$ and $\tau_2$ are two lower semi-continuous densely defined traces on $S$. Then $\tau_1 \otimes \tau_U$ and $\tau_2 \otimes \tau_U$ are each lower semi-continuous densely defined traces on the Pedersen ideal of $S\otimes U$. It follows that $\tau_1 \otimes \tau_U= c \tau_2 \otimes \tau_U$ for some $c>0$. From which we obtain that $\tau_1= c \tau_2$. The proof for $U$ follows from the result for $S$.

The second part of the theorem follows from the first part of the theorem using \cite[Proposition 2.2 and Corollary 2.4]{ComZet} and the fact that $pAp$ is Morita equivalence to $A$ because $A$ is simple.
\end{proof}

\begin{prop}
The traces $\tau_H$ and $\tau_S\otimes \tau_U$ coincide under the Morita equivalence defined from the imprimitivity bimodule ${}_{S\otimes U}C^*(Z(P,Q))_H$ (see Subsection \ref{moritasubssec}). In particular, the following diagram commutes:
\begin{equation}
\label{commdiagindprob}
\xymatrix{
K_0(H) \ar[rdd]_{{\rm K_0(\tau_H)}}  
 \ar[rr]^{\cong} 
 & & K_0(S\otimes U) \ar[ldd]^{{\rm K_0(\tau_S)}\otimes {\rm K_0(\tau_U)}}
 \\ \\
& \field{R} &
}
\end{equation}
where the horizontal arrow is Morita equivalence.
\end{prop}

\begin{proof}
Since $S$, $U$, and $H$ carry unique traces, by Proposition \ref{uniqueTrace}, it suffices to prove that the diagram commutes when applied to $[1]\in K_0(H)$. Indeed, uniqueness of traces implies that for a $\nu>0$ we have that $\tau_H=\nu \tau_S\otimes \tau_U$, so $\tau_H(1)=\nu (\tau_S\otimes \tau_U)(p)$. Here $p\in S\otimes U$ denotes a projection corresponding to $1\in H$ under the Morita equivalence, see Subsection \ref{moritasubssec}. 

We thus need to prove that $\tau_S\otimes \tau_U(p)=1$. For notational convenience, we let $G_s^{(0)}$ and $G_u^{(0)}$ denote the unit spaces $X^u(P)$ and $X^s(Q)$, respectively, viewed as subspaces of $G_s$ and $G_u$. It is readily verified that 
$$V_{ij}\cap (G_s^{(0)}\times G_u^{(0)})=
\begin{cases}
\emptyset, \; &\mbox{if $i\neq j$},\\
\{((x_1,x_1),(x_1',x_1')): \; x_1\in X^u(g_i,\epsilon), \;x_1'\in X^s(g_i,\epsilon)\}, \; &\mbox{if $i= j$}.
\end{cases}$$
We can therefore compute using Equation \eqref{productform} that 
\begin{align*}
\tau_S\otimes\tau_U(p)&=\sum_{i=1}^K \int_{X^u(P)\times X^s(Q)} f_i([x,x'])^2\ \mathrm{d}(\mu_u^P\times \mu_s^Q)(x,x')\\
&=\sum_{i=1}^K \int_{X^u(g_i,\epsilon)\times X^s(g_i,\epsilon))} f_i([x,x'])^2\ \mathrm{d}(\mu_u^P\times \mu_s^Q)(x,x')\\
&=\sum_{i=1}^K \int_{[X^u(g_i,\epsilon), X^s(g_i,\epsilon))]} f_i(y)^2\ \mathrm{d}\mu(y)=\int_{X} \sum_{i=1}^K f_i(y)^2\ \mathrm{d}\mu(y)=1.
\end{align*}
In the last equality, we used that $\sum_{i=1}^K f_i^2=1$ which follows from the fact that $\{f_1,\ldots, f_K\}$ is a partition of unity.
\end{proof}

\begin{prop} 
\label{denseReal}
The range of $K_0({\rm \tau_H})$ is dense in $\R$.
\end{prop}

\begin{proof}
Let $\alpha_S$ denote the automorphism on $S$ induced by $\varphi$. Then $\alpha_S \otimes  \mathrm{id}_U$ is an automorphism of $S\otimes U$. Moreover, since $\tau_S (\alpha_S(b))=\lambda \tau_S(b)$ for each $b \in {\rm Dom}(\tau_S)$, we have that for any $a$ with finite trace,
\[ (\tau_S \otimes \tau_U)((\alpha_S \otimes \mathrm{id}_U)(a)) = \lambda (\tau_S \otimes \tau_U)(a) \]
where $\log(\lambda)$ is the topological entropy. 

Using the previous proposition and the fact that $\tau_H[1_H]=1$, we have that the range of $K_0({\rm \tau_H})$ contains the subgroup of $\R$ generated by $\{ \lambda^{-n} \: | \: n \in \Z \}$. Since $\lambda>1$, the range of $K_0({\rm \tau_H})$ is dense in $\R$. 
\end{proof}

In our proofs concerning real rank zero, we will need a result of R\o rdam:

\begin{lemma} (see Corollary 7.3 in \cite{Ror}) \label{RorResult} \\
If $A$ is a simple, unital,  exact, $\mathcal{Z}$-absorbing $C^*$-algebra with unique trace $\tau$, then $A$ has real rank zero if and only if range of $K_0({\rm \tau})$ is dense in $\R$.
\end{lemma}

\begin{theorem} 
\label{RRZeroHom}
Suppose $(X, \varphi)$ is a mixing Smale space and $H$, $S$ and $U$ be the associated homoclinic, stable and unstable algebras. Also let $H\rtimes \Z$ denote the homoclinic algebra crossed with the integer action induced by $\varphi$. Then $H$, $S$, $U$ and $H\rtimes \Z$ each have real rank zero.
\end{theorem}

\begin{proof}
We begin with $H$ and show that the hypotheses of Lemma \ref{RorResult} are satisfied. The homoclinic groupoid is amenable so the associated $C^*$-algebra is nuclear and hence exact. That $(X,\varphi)$ is mixing implies that $H$ is simple is shown in \cite{PutSpi}. Finally $H$ is $\mathcal{Z}$-stable by the main result of \cite{DS}. Proposition \ref{denseReal} shows that the range of $K_0({\rm \tau_H})$ is dense in $\R$. Thus we can apply Lemma \ref{RorResult}, which implies that $H$ has real rank zero. 

Since the stable algebra of $(X, \varphi^{-1})$ is the unstable algebra of $(X, \varphi)$ we need only prove the result for $S$. Let $p\neq 0$ be a projection in $S$. Then by Proposition \ref{uniqueTrace} $pSp$ has a unique tracial state. Furthermore it is simple, unital, nuclear and $\mathcal{Z}$-stable so Lemma \ref{RorResult} reduces the proof to showing that the range of the trace is dense in $\R$. This follows as in the proof of Proposition \ref{denseReal} using the fact that $\tau_S (\alpha_S(b))=\lambda \tau_S(b)$ for each $b \in {\rm Dom}(\tau_S)$ where $\lambda>1$.

For $H\rtimes \Z$ the argument is similar, again one checks that the hypotheses Lemma \ref{RorResult} are satisfied; the details are omitted.
\end{proof}

Smale's decomposition theorem allows one to generalize many results from the mixing case to the nonwandering case. For completeness we include the details in the case of real rank zero.

\begin{theorem} 
\label{RRZeroNonWanCase}
Suppose $(X, \varphi)$ is a nonwandering Smale space and $H$, $S$ and $U$ be the associated homoclinic, stable and unstable algebras. Then each of these algebras have real rank zero.
\end{theorem}

\begin{proof}
We only give the details for the homoclinic algebra as the proof in each case is similar. The $C^*$-algebraic consequence of Smale's decomposition theorem is that $H$ is isomorphic to the finite direct sum of the homoclinic algebras associated to mixing Smale spaces, see \cite[Section 2]{PutFunProp}. Using the previous theorem and the fact that real rank zero respects direct sums, we have that $H$ has real rank zero.
\end{proof}

\end{document}